\documentclass[12pt]{amsart}
\usepackage{amsmath, amsthm, amscd, amsfonts, amssymb, graphicx, color}
\usepackage{enumerate}
\usepackage[bookmarksnumbered, colorlinks, plainpages]{hyperref}
\usepackage{tikz}
\usepackage{comment}

\newtheorem{theorem}{Theorem}[section]
\newtheorem{lemma}[theorem]{Lemma}
\newtheorem{proposition}[theorem]{Proposition}

\newtheorem{corollary}[theorem]{Corollary}
\newtheorem{definition}[theorem]{Definition}

\newtheorem{remark}[theorem]{Remark}

\excludecomment{mysection}

\DeclareMathOperator{\co}{co}

\begin{document}
	
\title{On the geometry of an order unit space}
\author{Anil Kumar Karn}
	
\address{School of Mathematical Sciences, National Institute of Science Education and Research Bhubaneswar, An OCC of Homi Bhabha National Institute, P.O. - Jatni, District - Khurda, Odisha - 752050, India.}

\email{\textcolor[rgb]{0.00,0.00,0.84}{anilkarn@niser.ac.in}}

	
	
\begin{abstract}
	We introduce the notion of \emph{skeleton} with a head in a non-zero real vector space. We prove that skeletons with heads describe order unit spaces geometrically. Next, we consider the notion of \emph{periphery} corresponding to an order unit space which is a part of the skeleton. We note that periphery consists of boundary elements of the positive cone with unit norms. We discuss some elementary properties of the periphery. We also find a condition under which $V$ would contain a copy of $\ell_{\infty}^n$ for some $n \in \mathbb{N}$ as an order unit subspace. 
\end{abstract}

\maketitle 

\section{Introduction} 

	Let $X$ be a normed linear space and let $x, y \in X$. We say that $x$ is \emph{$\infty$-orthogonal} to $y$, (we write, $x \perp_{\infty} y$), if $\Vert x + k y \Vert = \max \lbrace \Vert x \Vert, \Vert k y \Vert \rbrace$ for all $k \in \mathbb{R}$. It was proved in \cite{K14} that if $(V, e)$ is an order unit space and if $u, v \in V^+ \setminus \lbrace 0 \rbrace$, then $u \perp_{\infty} v$ if and only if $\Vert \Vert u \Vert^{- 1} u + \Vert v \Vert^{- 1} v \Vert = 1$. For $u, v \in V^+$, we say that $u$ is \emph{absolutely $\infty$-orthogonal} to $v$ (we write $u \perp_{\infty}^a v$) if $u_1 \perp_{\infty} v_1$ whenever $0 \le u_1 \le u$ and $0 \le v_1 \le v$.  

	Let $A$ be a unital C$^*$-algebra. Then $p \in A$ is a projection if $p^2 = p = p^*$ or equivalently, $p, 1 - p \in A^+$ and $p (1 - p) = 0$. Following \cite[Theorem 2.1]{K19}, we note that $p$ is a projection if and only if $p, 1 - p \in A^+$ and $p \perp_{\infty}^a (1 - p)$. In this paper, we weaken the notion of projections and consider in stead the notion of \emph{peripheral elements}. Let $(V, e)$ be an order unit space. An element $u \in V$ is said to be a peripheral element if $u, e - u \in V^+$ and we have $u \perp_{\infty} (e - u)$. The set of all peripheral elements together with $0$ and $e$ form the notion of \emph{skeleton with a head} in an order unit space in the following sense. 
\begin{definition}\label{skel}
	Let $X$ be a non-zero real vector space and let $S \subset X$ containing $0$ and $e \ne 0$. We say that $S$ is a \emph{skeleton} with $e$ as its \emph{head}, if the following properties hold. 
	\begin{enumerate}
		\item If $u \in S$, then $e - u \in S$; 
		
		Put $S_0 = S \setminus \lbrace 0, e \rbrace$. 
		\item If $u, v \in S$ and $\lambda \in [0, 1]$, then there exist $w \in S_0$ and $\alpha, \beta \in \mathbb{R}^+$ such that $\lambda u + (1 - \lambda) v = \alpha e + \beta w$; and
		\item If $e = \sum_{i=1}^n \alpha_i u_i$ for some $u_1, \dots, u_n \in S_0$ and $\alpha_1, \dots, \alpha_n > 0$ with $n \ge 2$, then $\sum_{i \ne j} \alpha_i \ge 1$ for all $j = 1, \dots, n$. 
	\end{enumerate} 
	The members of $S_0$ are called the \emph{peripheral} elements of $S$ and $S_0$ is called the \emph{periphery} of $S$. 
\end{definition} 
	Order unit spaces dominate the interface of commutative and non-commutative C$^*$-algebras. In early 1940's, Stone, Kakutani, Krein and Yosida proved independently that if an order unit space $(V, e)$ is a vector lattice in its order structure, then it is unitally lattice isomorphic to a dense lattice subspace of $C_{\mathbb{R}}(X)$ for some suitable compact Hausdorff space $X$ \cite[Theorem II.1.10]{A71}. (see the notes after Section1, Chapter II of \cite{A71} for the details.) In 1951, Sherman proved that the self-adjoint part of a C$^*$-algebra $A$ is a vector lattice in its order structure if and only if $A$ is commutative \cite{S51}. The same year, Kadison prove that the infimum of a pair of self-adjoint operators on a complex Hilbert space exists if and only if they are comparable \cite{K51a}. The same year in another paper, he proved that any unital self-adjoint subspace of a unital C$^*$-algebra is an order unit space \cite{K51b}. (Much later in 1977, Choi and Effros proved that a unital self-adjoint subspace of a unital C$^*$-algebra is precisely a matrix order unit space \cite{CE77}.)

	Soon after Kadison underscored the importance of order unit spaces as a possible role model for a non-commutative ordered spaces, there was a flux of research in the study of order unit spaces and their duals. Some early prominent references are Bonsall, Edwards, Ellis, Asimov and Ng, besides many others. (See  \cite{A68,B55,E64,El64,N69}. We refer to  \cite{A71,Jam70} for more references and details.) 

	The dual of an order unit space is a base normed space which is defined through the geometric notion of a \emph{base} in an ordered vector space. On the other hand, the notion of an order unit is order theoretic. In this paper we propose to study a set of geometric properties that determine order unit spaces. More precisely, in the following result we show that \emph{skeletons} describe order unit spaces geometrically. 
\begin{theorem}\label{skelg}
	Let $X$ be a non-zero real vector space and let $S$ be a skeleton in $X$ with $e$ as its head for some $e \in X$ with $e \ne 0$. Let $V$ be the linear span of $S$ and let $V^+$ be the cone generated by $S$. Then $(V, V^+, e)$ is an order unit space such that 
	$$S_0 := S \setminus \lbrace 0, e\rbrace = \lbrace v \in V^+: \Vert v \Vert = \Vert e - v \Vert = 1 \rbrace.$$ 
	(Here $\Vert\cdot\Vert$ is the order unit norm on $V$.)
\end{theorem} 
We also prove the converse of this result. 
\begin{theorem}\label{oskel}
	Let $(V, e)$ be an order unit space. Put 
	$$(S_V)_0 := C_V \cap (e - C_V) = \lbrace u \in V: \Vert u \Vert = \Vert e - u \Vert = 1 \rbrace.$$ 
	Then $S_V := (S_V)_0 \bigcup \lbrace 0, e \rbrace$ is a skeleton in $V$ with $e$ as its head such that $(V, e)$ is the order unit space generated by $S_V$. 
\end{theorem} 

	Next, we discuss the \emph{periphery} corresponding to an order unit space. We find that the periphery is consists of maximal elements of a canopy in a certain sense. The periphery includes projections whenever they exist. We discuss some elementary properties of the periphery. Using these properties, we prove that any order unit space $(V, e)$ of dimension more than $1$ contains a copy of $\ell_{\infty}^2$ as an order unit subspace. We also prove that $V$ is a union these copies in such a way any two such subspace meet at the \emph{axis} $\mathbb{R} e$. Further we find a condition under which $V$ would contain a copy of $\ell_{\infty}^n$ for some $n \in \mathbb{N}$ as an order unit subspace. 

	The scheme of the paper is as follows. In Section 2, we discuss some of the properties of \emph{skeleton} in a non-zero real vector space and prove Theorem \ref{skelg}. In Section 3, we prove Theorem \ref{oskel}. In Section 4, we study some elementary properties of the periphery corresponding to an order unit space. In Section 5, we find a condition under which an order unit space would contain a copy of $\ell_{\infty}^n$ for some $n \in \mathbb{N}$ as an order unit subspace besides some other results. 

\section{The skeleton} 

	In this section we shall prove Theorem \ref{skelg}. We begin with some preliminary results. Throughout in this section, we shall assume that $X$ is a non-zero real vector space and $S$ is a skeleton in $X$ with $e \ne 0$ as its head (see Definition \ref{skel}). First of all, we prove some of the easy consequences of Definition \ref{skel}. 
\begin{lemma}\label{propS}
	\begin{enumerate}
		\item $S_0 \bigcap [0, 1] e = \emptyset$.\label{skl1}
		\item \label{rim2} If $u, v \in S_0$ and $\alpha \in [0, 1]$, then there exist $w \in S_0$ and $\lambda, \mu \in \mathbb{R}^+$ with $\lambda \le \min \lbrace \alpha, 1 - \alpha \rbrace$ and $\lambda + \mu \le 1$ such that $\alpha u + (1 - \alpha) v = \lambda e + \mu w$. 
		\item Let $u, v \in S_0$ be such that $e = \alpha u + \beta v$ for some $\alpha, \beta \in \mathbb{R}^+$. Then $\alpha = 1 = \beta$. \label{n1}
		\item Let $u, v \in S_0$ such that $\alpha u = \beta v$ for some $\alpha, \beta \in \mathbb{R}$. \label{li}
		\begin{enumerate}[$(a)$] 
			\item Then $\alpha \beta \ge 0$;
			\item $\alpha = 0$ if and only if $\beta = 0$; 
			\item If $\alpha \beta > 0$. Then $u = v$. 
		\end{enumerate}
		\item Let $u_1, \dots, u_n$ be distinct elements of $S_0$ and $\alpha_0, \alpha_1, \dots, \alpha_n \in \mathbb{R}^+$ such that $\alpha_0 e + \sum_{i=1}^n \alpha_i u_i = 0$. Then $\alpha_i = 0$ for each $i = 0, 1, \dots, n$. \label{pos} 
	\end{enumerate} 
\end{lemma} 
\begin{proof}
	$(1)$ Let $\alpha e \in S_0$ for some $\alpha \in (0, 1)$. Then by \ref{skel}(1), $(1 - \alpha) e \in S_0$. Since $e = \alpha e + (1 - \alpha) e$, by \ref{skel}(3), we get $\alpha \ge 1$ and $1 - \alpha \ge 1$ which is absurd. 
	
	$(2)$ By \ref{skel}(2), we have $\alpha u + (1 - \alpha) v = \lambda e + \mu w$ for some $w \in $ and $\lambda, \mu \in \mathbb{R}^+$. Thus $(\lambda + \mu) e = \alpha u + (1 - \alpha) v + \mu (e - w)$. Now by \ref{skel}(3), we get $\lambda + \mu \le 1$, $\lambda + \mu \le \alpha + \mu$ and $\lambda + \mu \le 1 - \alpha + \mu$. Thus $\lambda \le \min \lbrace \alpha, 1 - \alpha \rbrace$ and $\lambda + \mu \le 1$. 
	
	$(3)$ By \ref{skel}(3), we have $\alpha \ge 1$ and $\beta \ge 1$. Now, $\alpha (e - u) + \beta (e - v) = (\alpha + \beta - 1) e$ and by \ref{skel}(1), $e - u, e - v \in S_0$. Thus invoking \ref{skel}(3) once again, we get $\alpha \ge \alpha + \beta - 1$ and $\beta \ge \alpha + \beta - 1$. Thus $\alpha \le 1$ and $\beta \le 1$ so that $\alpha = 1 = \beta$. 
	
	$(4)(a)$ Let $\alpha \beta < 0$. For definiteness, we assume that $\alpha > 0$ and $\beta < 0$. Then $\alpha (e - u) - \beta (e - v) = (\alpha - \beta) e$ with $\alpha - \beta > 0$. Thus by condition \ref{skel}(3), we get $\alpha, - \beta \ge \alpha - \beta$. But then $\alpha = 0 = \beta$ which contradicts the assumption. Thus $\alpha \beta \ge 0$. 
	
	$(4)(b)$ follows immediately as $0 \notin S_0$. 
	
	$(4)(c)$ Now assume that $\alpha \beta > 0$. For definiteness, we assume that $\alpha > 0$ and $\beta > 0$. Further, without any loss of generality, we may assume that $\alpha \le \beta$. Put $\frac{\alpha}{\beta} = \lambda$. Then $0 < \lambda \le 1$ and $\lambda u = v$. Thus $e = \lambda u + (e - v)$. Now by \ref{skel}(3), we get $\lambda \ge 1$ so that $\lambda = 1$. Thus $u = v$. 
	
	$(5)$ We have $\sum_{i=1}^n \alpha_i (e - u_i) = (\sum_{i=0}^n \alpha_i) e$. Assume, if possible that $\sum_{i=0}^n \alpha_i > 0$. Then by \ref{skel}(3), we get $\sum_{i=0}^n \alpha_i \le (\sum_{i=1}^n \alpha_i) - \alpha_j$ for all $j = 1, \dots, n$. In other words, $\alpha_0 + \alpha_j = 0$ for all $j = 1, \dots, n$. Therefore, $\alpha_j = 0$ for every $j = 0, 1, \dots, n$. 
\end{proof}
\begin{lemma}\label{unique}
	Let $u, v \in S_0$ be such that $\alpha e + \beta u = \gamma e + \delta v$ for some $\alpha, \beta, \delta, \gamma \in \mathbb{R}$ with $\beta, \delta \ge 0$. Then $\alpha = \gamma$ and we have either $\beta = 0 = \delta$ or $u = v$. 
\end{lemma} 
\begin{proof}
	First, we show that $\alpha = \gamma$. If $\alpha < \gamma$, then $(\gamma - \alpha + \delta) e = \beta u + \delta (e - v)$. By Lemma \ref{propS}(\ref{pos}), we must have $\gamma - \alpha + \delta > 0$ as $\delta \ge 0$. Thus by \ref{skel}(3), we have $\gamma - \alpha + \delta \le \delta$. But then we arrive at a contradiction, $\gamma \le \alpha$. Thus $\alpha \ge \gamma$. Now, by symmetry, we have $\gamma \ge \alpha$ so that $\alpha = \gamma$. Thus $\beta u = \delta v$. The rest of the proof follows from Lemma \ref{propS}(\ref{li}). 
\end{proof}
\begin{proposition}\label{intrc}
	For $u \in S_0$ we consider 
	$$K(u) := \co \lbrace 0, e, u \rbrace = \bigcup_{\alpha \in [0, 1]} \alpha [e, u].$$ 
	\begin{enumerate}[$(a)$]
		\item For $u, v \in S_0$ with $u \ne v$, we have $K(u) \cap K(v) = [0, 1] e$. 
		\item For $u \in S_0$, we have $K(e - u) = e - K(u)$.
	\end{enumerate}
\end{proposition} 
\begin{proof}
	$(a)$ Let $w \in K(u) \cap K(v)$. Then there exist $\alpha, \beta, \gamma, \delta \in [0, 1]$ with $\alpha + \beta \le 1$ and $\gamma + \delta \le 1$ such that $w = \alpha e + \beta u = \gamma e + \delta v$. If possible, assume that $\alpha \ne \gamma$. For definiteness, we let $\alpha > \gamma$. Then $(\alpha - \gamma + \beta) e = \beta (e - v) + \delta v$. Thus by \ref{skel}(3), we get $\beta \ge \alpha - \gamma + \beta$ so that $\alpha \le \gamma$. This contradicts the assumption. Hence $\alpha = \gamma$ so that $\beta u = \delta v$. Now by Lemma \ref{propS}(\ref{li}), we have $\beta = 0 = \delta$ as $u \ne v$. Thus $w \in [0, 1] e$, that is, $K(u) \cap K(v) \subset [0, 1] e$. As $[0, 1] e \subset K(x)$ for any $x \in S_0$, the proof is complete. 
	
	$(b)$ Let $w \in K(e - u)$. Then $w = \alpha e + \beta (e - u)$ for some $\alpha, \beta \in [0, 1]$ with $\alpha + \beta \le 1$. Thus $e - w = (1 - \alpha - \beta) e + \beta u$. Since $1 - \alpha - \beta, \beta \in [0, 1]$ and $1 - \alpha - \beta + \beta = 1 - \alpha \le 1$, we get that $e - w \in K(u)$. Thus $K(e - u) \subset K(u)$ for all $u \in S_0$. So for any $u \in S_0$, we also have $K(u) = K(e - (e - u)) \subset K(e - u)$. Therefore, $K(e - u) = K(u)$ for all $u \in S_0$. 
\end{proof} 
\begin{corollary}\label{intrr}
	For $u, v \in S_0$ with $u \ne v$, we have $[e, u] \bigcap [e, v] = \lbrace e \rbrace$.
\end{corollary} 
\begin{proof}
	Note that if $\alpha e \in [e, u]$ for some $u \in S_0$ and $\alpha \in [0, 1]$, say, $\alpha e = (1 - \lambda) e + \lambda u$, then $\lambda (e - u) = (1 - \alpha) e$. As $e - u \in S_0$ by \ref{skel}(1) and as $S_0 \bigcap \mathbb{R} e = \emptyset$, we must have $\lambda = 0 = 1 - \alpha$. Thus $[e, u] \bigcap [e, v] = \lbrace e \rbrace$. Now as $[e, u] \subset K(u)$, the result follows from Proposition \ref{intrc}.
\end{proof} 
Let $E$ be a convex subset of a real vector space $X$ with $0 \in E$. An element $x \in E$ is called a \emph{lead point} of $E$, if for any $y \in E$ and $\lambda \in [0, 1]$ with $x = \lambda y$, we have $\lambda = 1$ and $y = x$. The set of all lead points of $E$ is denoted by $Lead(E)$. 

A non-empty set $E$ of a real vector space $V$ is said to be \emph{linearly compact}, if for any $x, y \in E$ with $x \ne y$, we have, the intersection of $E$ with the line through $x$ and $y$, $\lbrace \lambda \in \mathbb{R}: (1 - \lambda) x + \lambda y \in E \rbrace$, is compact (in $\mathbb{R}$). Note that if $E$ is convex, the above intersection is an interval. Following \cite[Proposition 3.2]{GK20}, we may conclude that if $E$ is a linearly compact convex set with $0 \in E$, then $Lead(E)$ is non-empty and for each $x \in E, x \ne 0$, there exist a unique $u \in Lead(E)$ and a unique $0 < \alpha \le 1$ such that $x = \alpha u$.

\begin{theorem}\label{convK}
	Let $X$ be a non-zero real vector space and let $S$ be a skeleton in $X$ with $e \ne 0$ as its head. Consider $K = \bigcup_{u \in S_0} K(u)$ and $C = \bigcup_{u \in S_0} [e, u]$. Then $K$ is convex set containing $0$ and $e$ such that $Lead(K) = C$. Moreover, $e$ is an extreme point of $K$.
\end{theorem} 
\begin{proof}
	Let $x, y \in K$. Then there are $u, v \in S_0$ and $\alpha, \beta, \gamma, \delta \in [0, 1]$ with $\alpha + \beta \le 1$ and $\gamma + \delta \le 1$ such that $x = \alpha e + \beta u$ and $y = \gamma e + \delta v$. Then for $\lambda \in [0, 1]$, we have 
	\begin{eqnarray*}
		\lambda x + (1 - \lambda) y &=& (\lambda \alpha + (1 - \lambda) \gamma) e + \lambda \beta u + (1 - \lambda) \delta v \\
		&=& (\lambda \alpha + (1 - \lambda) \gamma) e + k (\lambda_1 u + (1 - \lambda_1) v)
	\end{eqnarray*} 
	where $k := \lambda \beta + (1 - \lambda) \delta$ and $\lambda_1 := \frac{\lambda \beta}{\lambda \beta + (1 - \lambda) \delta} \in [0, 1]$. By Lemma \ref{propS}(\ref{rim2}), we can find $w \in S_0$ and $\eta, \kappa \in [0, \frac 12]$ with $\eta \le \min \lbrace \lambda_1, 1 - \lambda_1 \rbrace$ and $\eta + \kappa \le 1$ such that $\lambda_1 u + (1 - \lambda_1) v = \eta e + \kappa w$. Thus 
	$$\lambda x + (1 - \lambda) y = (\lambda \alpha + (1 - \lambda) \gamma) e + k (\eta e + \kappa w) = \alpha_1 e + \beta_1 w$$ 
	where 
	$$\alpha_1 = \lambda \alpha + (1 - \lambda) \gamma + k \eta = \lambda \alpha + (1 - \lambda) \gamma + (\lambda \beta + (1 - \lambda) \delta) \eta$$ 
	and 
	$$\beta_1 = k \kappa = (\lambda \beta + (1 - \lambda) \delta) \kappa.$$ 
	Now $\alpha_1 \ge 0$, $\beta_1 \ge 0$ and 
	\begin{eqnarray*}
		\alpha_1 + \beta_1 &=& \lambda \alpha + (1 - \lambda) \gamma + (\lambda \beta + (1 - \lambda) \delta) (\eta + \kappa) \\
		&\le& \lambda \alpha + (1 - \lambda) \gamma + \lambda \beta + (1 - \lambda) \delta \\
		&=& \lambda (\alpha + \beta) + (1 - \lambda) (\gamma + \delta) \\ 
		&\le& 1
	\end{eqnarray*} 
	for $\alpha + \beta \le 1$, $\gamma + \delta \le 1$ and $\eta + \kappa \le 1$. Thus $\lambda x + (1 - \lambda) y \in K(w) \subset K$. Hence $K$ is a convex set containing $0$ and $e$. 
	
	Next, we show that $Lead(K) = C$. Let $v \in K$ $v \ne 0$. If $v = \alpha e$ then $0 < \alpha \le 1$ and $e \in C$. Now assume that $v \notin [0, 1] e$. Then there exists a unique $u \in S_0$ such that $v \in K(u)$. In other words, $v = \alpha e + \beta u$ for some $\alpha, \beta \in [0, 1]$ with $\alpha + \beta \le 1$. Since $v \ne 0$, we have $\beta > 0$. Set $w = (\alpha + \beta)^{-1} v$. Then $w \in C$ and $v = (\alpha + \beta) w$. Thus $K$ has a representation in $C$. We show that $C = Lead(K)$. 
	
	Let $u \in C$ and assume that $u = \alpha w$ for some $w \in K$ and $\alpha \in [0, 1]$. Let $u \in [e, v]$ for some $v \in S_0$, say $u = \lambda v + (1 - \lambda) e = e - \lambda (e - v)$. As $w \in K$, we have $w = \gamma e + \delta x$ for some $x \in S_0$ and $\gamma, \delta \in [0, 1]$ with $\gamma + \delta \le 1$. Then $e - \lambda (e - v) = \alpha (\gamma e + \delta x)$ or equivalently, $(1 - \alpha \gamma) e = \lambda (e - v) + \alpha \delta x$. Since $\alpha, \gamma \in [0, 1]$, we have $1 - \alpha \gamma \ge 1$. If $\alpha \gamma = 1$, then $\alpha = 1$ and we have $u = w$. So we assume that $1 - \alpha \gamma > 0$. Thus by \ref{skel}(3), we get $1 - \alpha \gamma \le \alpha \delta$. Therefore, $1 \le \alpha (\gamma + \delta) \le \alpha \le 1$ for $\gamma + \delta \le 1$. So we have $\alpha = 1$ and $u = w$ once again. Hence $C \subset Lead(K)$. 
	
	Conversely, let $u \in Lead(K)$. If $u = \alpha e$ for some $\alpha \in [0, 1]$, then by the definition of Lead, we have $\alpha = 1$ and $u = e \in C$. So we assume that $u \notin [0, 1] e$. Then as above, there exists $x \in C$ and $\lambda \in [0, 1]$ such that $u = \lambda x$. As $x \in Lead(K)$, we must have $\lambda = 1$ and $u = x \in C$. Hence $Lead(K) \subset C$ and consequently, $Lead(K) = C$. 
	
	Finally, we show that $e$ is an extreme point of $K$. Let $e = \alpha u + (1 - \alpha) v$ for some $u, v \in K$ and and for $0 < \alpha < 1$. Find $u_1, v_1 \in C$ and $\lambda_1, \mu_1 \in [0, 1]$ such that $u = \lambda_1 u_1$ and $v = \mu_1 v_1$. Then $e = \alpha \lambda_1 u_1 + (1 - \alpha) \mu_1 v_1$. Thus by Lemma \ref{propS}(\ref{n1}), we have $1 \le \alpha \lambda_1 + (1 - \alpha) \mu_1 \le 1$, that is, $\lambda_1 = 1 = \mu_1$. Therefore, $u, v \in C$. Find $x, y \in S_0$ and $\lambda, \mu \in [0, 1]$ such that $u = e - \lambda x$ and $v = e - \mu y$. Then $e = e - \alpha \lambda x - (1 - \alpha) \mu y$ whence $\alpha \lambda x + (1 - \alpha) \mu y = 0$. Now, by Lemma \ref{propS}(\ref{li}), we must have $\alpha \lambda = 0 = (1 - \alpha) \mu$. Since $0 < \alpha < 1$, we conclude that $\lambda = 0 = \mu$. Thus $u = e = v$ and consequently, we conclude that $e$ is an extreme point of $K$. 
\end{proof} 
We call $C$ the \emph{canopy} of $K$ with $e$ its \emph{summit}. 
\begin{remark}
	It is easy to note that $K = \co(S)$. Thus $K \cap \mathbb{R} e = [0, 1]e$. Also, by Proposition \ref{intrc}, we have $e - K = K$. Thus $0$ is also an extreme point of $K$.
\end{remark} 
\begin{proposition}\label{charK}
	Let $X$ be a non-zero real vector space and let $S$ be a skeleton in $X$ with $e \ne 0$ as its head and let $u \in S_0$. Then $\alpha e + \beta u \in K$ if and only if $\alpha, \alpha + \beta \in \mathbb{R}^+$ and $\max \lbrace \alpha, \alpha + \beta \rbrace \le 1$. 
\end{proposition} 
\begin{proof}
	Let $\alpha e + \beta u \in K$. If $\alpha e + \beta u = \lambda e$ for some $\lambda \in [0, 1]$, then $(\lambda - \alpha) e = \beta u$. As $u \in S_0$, we must have $\lambda = \alpha$ and $\beta = 0$. Thus $0 \le \alpha + \beta = \alpha \le 1$. So we assume that $\alpha e + \beta u \notin \mathbb{R} e$. Then by Theorem \ref{convK}, there exists a unique $x \in C$ and $0 < \lambda \le 1$ such that $\alpha e + \beta u = \lambda x$. Consequently, we can also find $w \in S_0$ and $1 \le \theta < 1$ such that $x = \theta e + (1 - \theta) w$. Thus $\alpha e + \beta u = \lambda \theta e + \lambda (1 - \theta) w$. Now we show that $\alpha \ge 0$.
	
	Assume, if possible, that $\alpha < 0$. Then $\lambda - \alpha > 0$ and we have 
	$$(\lambda - \alpha) e = (\lambda \theta + \lambda(1 - \theta) - \alpha) e = \beta u + \lambda (1 - \theta) (e - w).$$
	
	If $\beta \ge 0$, then by Lemma \ref{propS}(\ref{n1}), we have $\lambda - \alpha = \beta = \lambda (1 - \theta)$. But then $\alpha = \lambda \theta \ge 0$ which is a contradiction. Thus $\beta < 0$ and we have 
	$$(\lambda - \alpha - \beta) e = - \beta (e - u)  + \lambda (1 - \theta) (e - w).$$ 
	Again invoking Lemma \ref{propS}(\ref{n1}), we conclude that $\lambda - \alpha - \beta = - \beta = \lambda (1 - \theta)$. This leads to another contradiction $\alpha = \lambda \ge 0$. Hence $\alpha \ge 0$. 
	
	Next, we aim to prove that $\alpha + \beta \ge 0$ and assume to the contrary that $\alpha + \beta < 0$, that is, $0 \le \alpha < - \beta$. Now $(\alpha - \lambda \theta) = - \beta u + (1 - \lambda) \theta w$ so by Lemma \ref{propS}(\ref{pos}), we must have $\alpha - \lambda \theta > 0$ as $\beta < 0$. By Lemma \ref{propS}(\ref{n1}), we have $\alpha - \lambda \theta = - \beta = (1 - \lambda) \theta$. Thus $\alpha + \beta = \lambda \theta \ge 0$ contradicting the assumption $\alpha + \beta < 0$. Thus $\alpha + \beta \ge 0$. 
	
	Since $\alpha e + \beta u \in K \setminus \mathbb{R} e$, there exists a unique $w \in S_0$ such that $\alpha e + \beta u = \gamma e + \delta w$ where $\gamma, \delta \in \mathbb{R}^+$ with $\gamma + \delta \le 1$. 
	
	Let $\beta \le 0$. Then $(\alpha - \gamma) e = - \beta u + \delta w$ so by Lemma \ref{propS}(\ref{pos}), we must have $\alpha - \gamma \ge 0$. If $\alpha = \gamma$, then we further get $- \beta = 0 = \delta$ so that 
	$$0 \le \alpha + \beta = \alpha = \gamma = \gamma + \delta \le 1.$$ 
	If $\alpha > \gamma$, then by Lemma \ref{propS}(\ref{n1}), we get $\alpha - \gamma = - \beta = \delta$. Thus $\alpha = \gamma + \delta \le 1$ and $\alpha + \beta = \gamma \le \gamma + \delta \le 1$. 
	
	Next, let $\beta > 0$. Then $(\alpha + \beta - \gamma) e = \beta (e - u) + \delta w$. Thus by Lemma \ref{propS}(\ref{pos}), we get $\alpha + \beta - \gamma \ge 0$. If $\alpha + \beta = \gamma$, we further get $\beta = 0 = \delta$, contradicting $\beta > 0$. Thus $\alpha + \beta > \gamma$. Invoking Lemma \ref{propS}(\ref{n1}), we have $\alpha + \beta - \gamma = \beta = \delta$. Hence $0 \le \alpha \le \alpha + \beta = \gamma + \delta \le 1$. 
	
	Conversely, we assume that $0 \le \alpha, \alpha + \beta \le 1$. When $\beta \ge 1$, we have $\alpha e + \beta u \in K(u) \subset K$ for $\alpha, \beta \ge 0$ and $\alpha + \beta \le 1$. When $\beta < 0$, we can write $\alpha e + \beta u = (\alpha + \beta) e - \beta (e - u) \in K(e - u) \subset K$ for $\alpha + \beta, - \beta \ge 0$ and $\alpha = (\alpha + \beta) - \beta = \alpha \le 1$. 
\end{proof}
Now we prove the characterization of order unit spaces. 
\begin{proof}[Proof of Theorem \ref{skelg}]
	Put $C = \bigcup_{u \in S_0} [e, u]$ and $K = \co(S)$. Then by Theorem \ref{convK}, $K$ is a convex set containing $0$ with $C = Lead(K)$. Also, by Proposition \ref{intrc}, $K = \bigcup_{\alpha \in [0, 1]} \alpha C$. Thus $V^+ =  \bigcup_{\lambda \in \mathbb{R}^+} \lambda C = \bigcup_{n \in \mathbb{N}} n K$ and $V = V^+ - V^+$. We prove that 
	$$K = \lbrace v \in V^+: 0 \le v \le e \rbrace. \qquad (*)$$ 
	Let $v \in K$. Then $v \in V^+$ so that $0 \le v$. Let $v = \alpha u$ for some $u \in C$ and $\alpha \in [0, 1]$. If $u = e$, then $v = \alpha e \le e$. So let $u \ne e$. Then there exists $w \in S_0$ and $\lambda \in [0, 1]$ such that $u = (1 - \lambda) e + \lambda w$. Thus 
	$$e - v = (1 - \alpha) e + \alpha \lambda (e - w) \in K(w) \subset V^+$$ 
	for $e - w \in S_0$ and $1 - \alpha, \alpha \lambda \ge 0$ with $1 - \alpha + \alpha \lambda \le 1$. Therefore, $K \subset \lbrace v \in V^+: 0 \le v \le e \rbrace$. 
	
	Conversely, assume that $0 \le u \le e$. Then $u, e - u \in V^+$. Thus there exist $v, w \in C$ and $\alpha, \beta \ge 0$ such that $u = \alpha v$ and $e - u = \beta w$. Then $e = \alpha v + \beta w$ so by the definition of a canopy, we must have $\alpha, \beta \le 1$. Therefore, $u = \alpha v \in K$. Hence $(*)$ is proved. 
	
	Since $V = V^+ - V^+$, it follows from $(*)$ that $e$ is an order unit for $V$. We prove that $V^+$ is proper. Let $\pm u \in V^+$. Then there exist $v, w \in C$ and $\alpha, \beta \ge 0$ such that $u = \alpha v$ and $- u = \beta w$. Thus $\alpha v + \beta w = 0$. We show that $\alpha = 0 = \beta$. Assume, if possible, that $\alpha > 0$. Then $\beta > 0$ too, for $v \ne 0$. Put $k = \frac{\alpha}{\alpha + \beta}$. Then $0 < k < 1$ and we have $ku + (1 - k) v = 0$ so that $e = k (e - u) + (1 - k) (e - v)$. As $u, v \in C$, we have $e - u, e - v \in K$. As $e$ is an extreme point of $K$ by Theorem \ref{convK}, we deduce that $e - u = e = e - v$, that is, $u = 0 = v$ which is absurd. Thus $\alpha = 0$. Therefore, $V^+$ is proper.
	
	Next, we show that $V^+$ is Archimedean. Let $v \in V$ be such that $k e + v \in V^+$ for all $k > 0$. Then $v_n := \left(\frac{1}{1 + \Vert v \Vert}\right) \left( \frac{1}{n+1} e + \frac{n}{n+1} v\right) \in V^+$ for all $n \in \mathbb{N}$. Since 
	$$\Vert v_n \Vert \le \left( \frac{1}{1 + \Vert v \Vert} \right) \left( \frac{1 + n \Vert v \Vert}{1 + n} \right) < 1,$$ 
	we have $v_n \in K$ for every $n$. Thus by Proposition \ref{charK}, for each $n$, there exists a unique $u_n \in S_0$ and $\alpha_n, \beta_n \in \mathbb{R}^+$ with $\alpha_n + \beta \le 1$ such that $v_n = \alpha_n e + \beta_n u_n$. Set $u_1 := u, \alpha = 2 \alpha_1 (1 + \Vert v \Vert)$ and $ \beta = 2 \beta_1 (1 + \Vert v \Vert)$. Then $v_1 = \frac{e + v}{2 (1 + \Vert v \Vert)}$. Thus $e + v = \alpha e + \beta u$, or equivalently, $v = (\alpha - 1) e + \beta u$ where $\alpha, \beta \in \mathbb{R}^+$ with $\alpha + \beta \le 2 (1 + \Vert v \Vert)$. It follows that 
	\begin{eqnarray*}
		\alpha_n e + \beta_n u_n = v_n &=& \frac{1}{(1 + \Vert v \Vert)} \left(\frac{e + n v}{n + 1} \right) \\ 
		&=& \frac{(n \alpha - n + 1) e + n \beta u}{(n + 1) (1 + \Vert v \Vert)},
	\end{eqnarray*} 
	that is, $\alpha_n e + \beta_n u_n = \gamma_n e + \delta_n u$ for all $n$ where $\gamma_n = \frac{(n \alpha - n + 1)}{(n + 1) (1 + \Vert v \Vert)}$ and $\delta_n = \frac{ + n \beta}{(n + 1) (1 + \Vert v \Vert)}$. Let $n \in \mathbb{N}$. Since $\beta_n \ge 0$ and $\delta_n \ge 0$, by Lemma \ref{unique} we have $\frac{(n \alpha - n + 1)}{(n + 1) (1 + \Vert v \Vert)} = \gamma_n = \alpha_n \in [0, 1]$. Taking limit as $n \to \infty$, we may conclude that $0 \le \alpha - 1 \le 1 + \Vert v \Vert$. Thus $v = (\alpha - 1) e + \beta u \in V^+$. 	 Therefore $V^+$ is Archimedean. 
	
	Now it follows that $(V, V^+, e)$ is an order unit space. We also note that $K = \lbrace v \in V^+: \Vert v \Vert \le 1 \rbrace$. Thus 
	$$C = Lead(K) = \lbrace v \in V^+: \Vert v \Vert = 1 \rbrace.$$ 
	We show that $S_0 = C \cap (e - C)$. Since $S_0 \subset C$ and since $e - S_0 = S_0$ by condition \ref{skel}(1), we have $S_0 \subset C \bigcap (e - C)$. Now let $w \in C \bigcap (e - C)$. Then $w, e - w \in C$. Thus there exist $u, v \in S_0$ and $\alpha, \beta \in [0, 1]$ such that 
	$$w = \alpha u + (1 - \alpha) e = e - \alpha (e - u)$$ 
	and 
	$$e - w = \beta v + (1 - \beta) e = e - \beta (e - v).$$ 
	Thus $e = \alpha (e - u) + \beta (e - v)$. Since $e - u, e - v \in R$, by condition \ref{skel}(3), we get $\alpha \ge 1$ and $\beta \ge 1$. Therefore, $\alpha = 1 = \beta$ so that $w = u \in S_0$. Hence $S_0 = \lbrace v \in V^+: \Vert v \Vert = 1 = \Vert e - v \Vert \rbrace$.
\end{proof} 

\section{The positive part of the closed unit ball} 
 We now discuss some properties of the positive elements with norm one in an order unit space.
\begin{proposition}\label{1}
	Let $(V, e)$ be an order unit space of dimension $\ge 2$. Consider $C_V = \lbrace u \in V^+: \Vert u \Vert = 1 \rbrace$.
	\begin{enumerate}
		\item Fix $u \in C_V$ with $u \ne e$ and consider the one dimensional affine subspace 
		$$L(u) = \lbrace u_{\lambda} := e - \lambda (e - u): \lambda \in \mathbb{R} \rbrace$$
		of $V$. Then  
		\begin{enumerate}
			\item $\lbrace u_{\lambda}: \lambda \in \mathbb{R} \rbrace$ is decreasing; 
			\item $u_{\lambda} \in V^+$ if and only if $\lambda \Vert e - u \Vert \le 1$;
			\item $\Vert u_{\lambda} \Vert = \max \lbrace 1, \vert \lambda \Vert e - u \Vert - 1 \vert \rbrace$ for every $\lambda \in \mathbb{R}$; 
			\item there exists a unique $\bar{u} \in L(u)$ such that $\bar{u}, e - \bar{u} \in C_V$.
		\end{enumerate}
		\item For $u, v \in C_V$, we have either $L(u) \bigcap L(v) = \lbrace e \rbrace$ or $L(u) = L(v)$.
	\end{enumerate} 
	
\end{proposition}
\begin{proof}
	(1)(a) follows from the construction of $u_{\lambda}$.
	
	(1)(b): Since $e - u \in V^+ \setminus \lbrace 0 \rbrace$, there exist $f_u \in S(V)$ such that $\Vert e - u \Vert = f_u(e - u) = 1 - f_u(u)$. Also then, $f_u(u) \le f(u)$ for all $f \in S(V)$ with $f_u(u) < 1$. Set  $\bar{\lambda} := \Vert e - u \Vert^{-1} = \left(\frac{1}{1 - f_u(u)}\right)$ and $\bar{u} := u_{\bar{\lambda}} = e - \bar{\lambda} (e - u)$. For $f \in S(V)$, we have  
	$$f(\bar{u}) = f(e) - \left( \frac{f(e) - f(u)}{1 - f_u(u)} \right) = \frac{f(u) - f_u(u)}{1 - f_u(u)} \ge 0$$
	so that $\bar{u} \in V^+$. Now by (1), $u_{\lambda} \in V^+$ if $\lambda \le \bar{\lambda}$. Also, if $u_{\lambda} \in V^+$ for some $\lambda \in \mathbb{R}$, then 
	$$0 \le f_u(u_{\lambda}) = f_u(e) - \lambda (f_u(e) - f_u(u)) = 1 - \lambda (1 - f_u(u)).$$ 
	Thus $\lambda \Vert e - u \Vert \le 1$.
	
	(1)(c): Fix $\lambda \in \mathbb{R}$.
	
	Case 1. $\lambda \ge 0$. Then for $k \in \mathbb{R}$, we have $u_{\lambda} \le k e$, that is, $\lambda u \le (k - 1 + \lambda) e$ if and only if $k \ge 1$. Next, for $l \in \mathbb{R}$, we have $l e + u_{\lambda} \in V^+$, that is, $(l + 1 - \lambda) e + \lambda u \in V^+$ if and only if $l + 1 - \lambda + \lambda f_u(u) \ge 0$ as $f_u(u) \le f(u)$ for all $f \in S(V)$. In other words, $l e + u_{\lambda} \in V^+$ if and only if $l \ge \lambda \Vert e - u \Vert - 1$. Thus for $\lambda \ge 0$, we have 
	$$\Vert u_{\lambda} \Vert = \inf \lbrace \alpha > 0: \alpha e \pm u_{\lambda} \in V^+ \rbrace = \max \lbrace 1, \lambda \Vert e - u \Vert - 1 \rbrace.$$
	
	Case 2. $\lambda < 0$. Then $l e + u_{\lambda} \in V^+$ for all $l \ge 0$. Next, for $k \in \mathbb{R}$, we have $u_{\lambda} \le k e$, that is, $(k - 1 + \lambda) e - \lambda u \ge 0$ if and only $k - 1 + \lambda - \lambda f_u(u) \ge 0$ for $f_u(u) \le f(u)$ for all $f \in S(V)$ and $- \lambda > 0$. Thus $u_{\lambda} \le k e$ if and only if $k \ge - \lambda \Vert e - u \Vert + 1$. Therefore, for $\lambda < 0$, we have 
	$$\Vert u_{\lambda} \Vert = \inf \lbrace \alpha > 0: \alpha e \pm u_{\lambda} \in V^+ \rbrace = 1 - \lambda \Vert e - u \Vert.$$
	
	Summing up, for any $\lambda \in \mathbb{R}$, we have  
	$$\Vert u_{\lambda} \Vert = \max \lbrace 1, \vert \lambda \Vert e - u \Vert - 1 \vert \rbrace.$$
	
	(1)(d): Put $\bar{u} = e - \Vert e - u \Vert^{-1} (e - u)$. Then by (c), $\Vert \bar{u} \Vert = 1$. Also by construction, $\Vert e - \bar{u} \Vert = 1$. Next, assume that $u_{\lambda} \in L(u)$ is such that $u_{\lambda}, e - u_{\lambda} \in C_V$. Then, as $\Vert e - u_{\lambda} \Vert = 1$, we get $\vert \lambda \vert \Vert e - u \Vert = 1$. If $\lambda = - \Vert e - u \Vert^{-1}$, then $\Vert u_{\lambda} \Vert = 2$ so we must have $\lambda = \Vert e - u \Vert^{-1}$. Thus $u_{\lambda} = \bar{u}$.
	
	(2): Let $w \in L(u) \bigcap L(v)$ with $w \ne e$. Then there are $\lambda, \mu \in \mathbb{R} \setminus \lbrace 0 \rbrace$ such that $w = e - \lambda (e - u) = e - \mu (e - v)$. Thus $\lambda (e - u) = \mu (e - v)$. Let $\alpha \in \mathbb{R}$. Then 
	$$e - \alpha (e - u) = e - \alpha \lambda \mu^{-1} (e - v)$$
	so that $L(u) \subset L(v)$. Now by symmetry, we have $L(u) = L(v)$.
\end{proof}
The following statements can be verified easily.
\begin{corollary}\label{2}
	Under the assumptions of Lemma \ref{1}, we have 
	\begin{enumerate}
		\item $\Vert u_{\lambda} \Vert = 1$ if and only if $0 \le \lambda \Vert e - v \Vert \le 2$;
		\item $\lbrace \Vert u_{\lambda} \Vert: \lambda \in (- \infty, 0] \rbrace$ is strictly decreasing; 
		\item $\lbrace \Vert u_{\lambda} \Vert: \lambda \in \left[ \frac{2}{\Vert e - u \Vert}, \infty \right) \rbrace$ is strictly increasing
		\item $C_V \bigcap (e - C_V) = \lbrace \bar{u}: u \in C_V \rbrace$; 
		\item $C(u) := L(u) \bigcap C_V = [e, \bar{u}]$;
		\item $L(u) \bigcap (e - C_V) = \lbrace \bar{u} \rbrace$.
	\end{enumerate}
\end{corollary} 
\begin{lemma}\label{periphery}
	Let $(V, e)$ be an order unit space and let $u \in C_V$. Then the following statements are equivalent:
	\begin{enumerate}
		\item $u \in (S_V)_0$; 
		\item $u \perp_{\infty} (e - u)$; 
		\item $u$ has an $\infty$-orthogonal pair in $C_V$;
		\item there exists $v \in C_V$ such that $u + v \in C_V$; 
		\item there exists a state $f$ of $V$ such that $f(u) = 0$.
	\end{enumerate} 
\end{lemma}
\begin{proof}
	If $u \in (S_V)_0$, then $u, e - u \in V^+$ with $\Vert u \Vert = 1 = \Vert e - u \Vert$. Also then $\Vert u + e - u \Vert = \Vert e \Vert = 1$ so that $u \perp_{\infty} (e - u)$. Thus $(1)$ implies $(2)$. Also, $(2)$ implies $(3)$ trivially. 
	
	Next, let $u \perp_{\infty} v$ for some $v \in C_V$. Then $\Vert u + v \Vert = 1$ so that $u + v \in C_V$. That is, $(3)$ implies $(4)$. 
	
	Now, if $u + v \in C_V$, then $u + v \le e$. Thus $v \le e - u \le e$ and we have $1 = \Vert v \Vert \le \Vert e - u \Vert \le \Vert e \Vert = 1$. Therefore, $u \in (S_V)_0$ so that $(4)$ implies $(1)$. 
	
	Again, if $u \in (S_V)_0$, then $e - u \in C_V$. Thus there exists a state $f$ of $V$ such that $1 = f(e - u) = 1 - f(u)$, or equivalently, $f(u) = 0$. Therefore, $(1)$ implies $(5)$. 
	
	Conversely, if $f(u) = 0$ for some state $f$ on $V$, then $f(e - u) = 1$ so that $\Vert e - u \Vert \ge 1$. Also, as $0 \le u \le e$, we have $0 \le e - u \le e$ so that $\Vert e - u \Vert \le 1$. Thus $e - u \in C_V$ whence $u \in (S_V)_0$. Hence $(5)$ implies $(1)$. 
\end{proof}
\begin{proof}[Proof of Theorem \ref{oskel}]
	
	It is enough to prove that $S_V$ is a skeleton in $V$ with $e$ as its head. By construction, we have $e - u \in S_V$ whenever $u \in S_V$. Also we note that $S_V \subset V^+$. In fact, if $u \in (S_V)_0$, then $u, e - u \le e$. 
	
	Let $u, v \in S_V$ and $\lambda \in (0, 1)$. Without any loss of generality, we may assume that $u, v \notin \lbrace 0, e \rbrace$. Then $x := \lambda u + (1 - \lambda) v \ne 0$. Put $x_1 = \Vert x \Vert^{-1} x$. Then $x_1 \in C_V$. Thus by Corollary \ref{2}(5), there exists a $w \in (S_V)_0$ and $\alpha_1 \in [0, 1]$ such that $x_1 = \alpha_1 e + (1 - \alpha_1) w$. Now it follows that  $\lambda u + (1 - \lambda) v = \alpha e + \beta w$ where $\alpha = \alpha_1 \Vert  \lambda u + (1 - \lambda) v \Vert$ and $\beta = (1 - \alpha_1) \Vert  \lambda u + (1 - \lambda) v \Vert$. Here $\alpha, \beta \in \mathbb{R}^+$ with $\alpha + \beta = \Vert  \lambda u + (1 - \lambda) v \Vert \le 1$. 
	
	Next, let $e = \sum_{i=1}^n \alpha_i$ for some $u_1, \dots, u_n \in (S_V)_0$ and $\alpha_1, \dots, \alpha_n > 0$ with $n \ge 2$. By Lemma \ref{periphery}, There exist $f_1, \dots, f_n \in S(V)$ such that $f_i(u_i) = 0$ for $i = 1, \dots, n$. Thus 
	$$1 = f_j(e) = f_j \left(\sum_{i=1}^n \alpha_i u_i \right) = \sum_{i \ne j} \alpha_i f_i(u_i) \le \sum_{i \ne j} \alpha_i.$$
\end{proof}

\section{The periphery}

Recall that if $(V, e)$ is an order unit space, then (the linearly compact set) $K_V := [0, e] = \lbrace v \in V^+: \Vert v \Vert \le 1 \rbrace$ has a canopy $C_V := \lbrace v \in V^+: \Vert v \Vert = 1 \rbrace$ with its summit at the order unit $e$. The periphery of $C_V$, denoted by $S_V$, is given by  
$$(S_V)_0 := C_V \cap (e - C_V) = \lbrace v \in V^+: \Vert v \Vert = 1 = \Vert e - v \Vert \rbrace.$$ 
In this section, we discuss some of the properties and examples of the periphery corresponding to an order unit space. 
\begin{theorem}
	Let $(V, e)$ be an order unit space and let $Bd(V^+)$ denote the  $\Vert\cdot\Vert$-boundary of $V^+$. Then $R_V = C_V \bigcap Bd(V^+)$.
\end{theorem}
\begin{proof}
	Let $u \in R_V$. Then by Lemma \ref{periphery}, $u \in C_V$ and there exists $f \in S(V)$ such that $f(u) = 0$. Thus for any $\epsilon > 0$, we have $f(u - \epsilon e) = - \epsilon < 0$ so that $u - \epsilon e \notin V^+$. Hence $u \in \overline{(V \setminus V^+)}$ and consequently, $u \in Bd(V^+)$. 
	
	Conversely, let $v \in Bd(V^+)$. Then $v \in V^+$ and there exists a sequence $\langle v_n \rangle$ in $V \setminus V^+$ such that $v_n \to v$. Let $n \in \mathbb{N}$. As $v_n \notin V^+$, we can find $f_n \in S(V)$ such that $f_n(v_n) < 0$. Since $S(V)$ is weak$^*$-compact, by passing to a subsequence, we may deduce that $f_n \to f_0 \in S(V)$ in the weak$^*$-topology. Now, as 
	\begin{eqnarray*}
		\vert f_n(v_n) - f_0(v) \vert &\le& \vert f_n(v_n) - f_0(v_n) \vert + \vert f_0(v_n - v) \vert \\ 
		&\le& \vert f_n(v_n) - f_0(v_n) \vert + \Vert v_n - v \Vert,
	\end{eqnarray*}
	we conclude that $f_0(v) = \lim f_n(v_n) \le 0$. Since $v \in V^+$ and since $f_0 \in S(V)$, we also have $f_0(v) \ge 0$. Thus $f_0(v) = 0$. Therefore, by Lemma \ref{periphery}, we have  $C_V \bigcap Bd(V^+) \subset R_V$ which completes the proof.
\end{proof}
\begin{remark}
	We have $Bd(V^+) = \bigcup_{u \in R_V} [0, \infty) u$. Also, for $u, v \in R_V$ with $u \ne v$, we have $[0, \infty) u \bigcap [0, \infty) v = \lbrace 0 \rbrace$. Thus the  $\Vert\cdot\Vert$-boundary of $V^+$ is the disjoint union of the rays passing through $R_V$.
\end{remark}
\begin{proposition}\label{disjR}
	Let $(V, e)$ be an order unit space and let $u, v \in (S_V)_0$ with $u \ne v$. Then the following statements are equivalent:
	\begin{enumerate}
		\item $(u, v) \bigcap (S_V)_0 \ne \emptyset$; 
		\item there are states $f$ and $g$ of $V$ such that $f(u) = 1 = f(v)$ and $g(u) = 0 = g(v)$;
		\item $[u, v] \subset (S_V)_0$.  
	\end{enumerate} 
\end{proposition}
\begin{proof}
	Let $w \in (u, v) \bigcap (S_V)_0$. Then $w = (1 - \alpha) u + \alpha v \in (S_V)_0$ for some $\alpha \in (0, 1)$. Find $f, g \in S(V)$ such that $f((1 - \alpha) u + \alpha v) = 1$ and $g((1 - \alpha) u + \alpha v) = 0$. Thus 
	$$1 = f((1 - \alpha) u + \alpha v) = (1 - \alpha) f(u) + \alpha f(v) \le (1 - \alpha) + \alpha = 1$$
	so that $f(u) = 1 = f(v)$. Again 
	$$0 = g((1 - \alpha) u + \alpha v) = (1 - \alpha) g(u) + \alpha g(v)$$ 
	so that $(1 - \alpha) g(u) = 0 = \alpha g(v)$. Since $0 < \alpha < 1$, we get $g(u) = 0 = g(v)$.
	
	Let $\lambda \in [0, 1]$ and consider $x = (1 - \lambda) u + \lambda v$. Then Then $x \in V^+$ with $\Vert x \Vert \le 1$. Also as 
	$$f(x) = (1- \lambda) f(u) + \lambda f(v) = 1 - \lambda + \lambda = 1,$$ 
	we note that $x \in C_V$. Further, 
	$$g(x) = (1- \lambda) g(u) + \lambda g(v) = 0$$ 
	so that $x \in (S_V)_0$.
\end{proof} 
\begin{corollary}
	Let $u, v \in (S_V)_0$ with $u \le v$. Then $[u, v] \subset (S_V)_0$.
\end{corollary} 
\begin{proof}
	Put $w = \frac 12 (u + v)$. Then $u \le w \le v$ and $e - v \le e - w \le e - u$. Since $u, v, e - u, e - v \in C_V$, we get that $w, e - w \in C_V$. Thus $ w \in (S_V)_0$. Now by Proposition \ref{disjR}, we may conclude that $[u, v] \subset (S_V)_0$. 
\end{proof}
Now the next result follows immediately.
\begin{corollary}
	Let $(V, e)$ be an order unit space and assume that $u_1, u_2 \in R_V$ with $u_1 \perp_{\infty} u_2$. Then 
	\begin{enumerate}
		\item $[u_1, e - u_2] \bigcup [u_2, e - u_1] \subset R_V$; 
		\item $\left((u_1, u_2) \bigcup (e - u_1, e - u_2)\right) \bigcap R_V = \emptyset$. 
	\end{enumerate}
\end{corollary}
\subsection{Direct sum of order unit spaces.} 
Next we turn to describe $C_{\ell_{\infty}^n}$ and $R_{\ell_{\infty}^n}$.
\begin{lemma}
	Let $(V_1, e_1)$ and $(V_2, e_2)$ be any two order unit spaces. Consider $V = V_1 \times V_2$, $V^+ = V_!^+ \times V_2^+$ and $e = (e_1, e_2)$. Then $(V, e)$ is also an order unit space and we have 
	\begin{enumerate}
		\item $C_V = (C_{V_1} \times [0, e_2]_o) \bigcup ([0, e_2]_o \times C_{V_2})$ and 
		\item $R_V = (R_{V_1} \times [0, e_2]_o) \bigcup ([0, e_2]_o \times R_{V_2}) \bigcup (C_{V_1} \times (e_2 - C_{V_2})) \bigcup ((e_1 - C_{V_1}) \times C_{V_2})$.
	\end{enumerate}
\end{lemma}
\begin{proof}
	For $(v_1, v_2) \in V$, we have $\Vert (v_1, v_2) \Vert = \max \lbrace \Vert v_1 \Vert, \Vert v_2 \Vert \rbrace$. Thus $(u_1, u_2) \in C_V$ if and only if $u_1 \in V_1^+$, $u_2 \in V_2^+$ and $\max \lbrace \Vert u_1 \Vert, \Vert u_2 \Vert \rbrace = 1$. Therefore, $C_V = (C_{V_1} \times [0, e_2]_o) \bigcup ([0, e_2]_o \times C_{V_2})$. Now, as $(u_1, u_2) \in R_V$ if and only $(u_1, u_2), (e_1 - u_1, e_2 - u_2) \in C_V$, we may deduce that $R_V = (R_{V_1} \times [0, e_2]_o) \bigcup ([0, e_2]_o \times R_{V_2}) \bigcup (C_{V_1} \times (e_2 - C_{V_2})) \bigcup ((e_1 - C_{V_1}) \times C_{V_2})$.
\end{proof}
Replace $V_2$ by $\mathbb{R}$. As $C_{\mathbb{R}} = \lbrace 1 \rbrace$ and $R_{\mathbb{R}} = \emptyset$, we may conclude the following:  
\begin{corollary}\label{oneup}
	Let $(V, e)$ be an order unit space. Consider $\hat{V} = V \times \mathbb{R}$, $\hat{V}^+ = V^+ \times \mathbb{R}^+$ and $\hat{e} = (e, 1)$. Then $(\hat{V}, \hat{e})$ is an order unit space and we have 
	\begin{enumerate}
		\item $C_{\hat{V}} = (C_V \times [0, 1]) \bigcup ([0, e_2]_o \times \lbrace 1 \rbrace)$ and 
		\item $R_{\hat{V}} = (R_V \times [0, 1]) \bigcup (C_V \times \lbrace 0 \rbrace) \bigcup ((e - C_V) \times \lbrace 1 \rbrace)$. 
	\end{enumerate}
\end{corollary}
Again using $C_1 = \lbrace 1 \rbrace$, $R_1 = \emptyset$ and following the induction on $n$, we can easily obtain the canopy and its periphery of $\ell_{\infty}^n$ with the help of Corollary \ref{oneup}.  
\begin{corollary}
	Fix $n \in \mathbb{N}$, $n \ge 2$. Put $C_n := C_{\ell_{\infty}^n}$ and $R_n := R_{\ell_{\infty}^n}$. Then 
	\begin{enumerate}
		\item $C_n = \lbrace (\alpha_1, \dots, \alpha_n): \min \lbrace \alpha_i \rbrace \ge 0 ~ \mbox{and} ~ \max \lbrace \alpha_i \rbrace = 1 \rbrace$ and
		\item $R_n = \lbrace (\alpha_1, \dots, \alpha_n): \min \lbrace \alpha_i \rbrace = 0 ~ \mbox{and} ~ \max \lbrace \alpha_i \rbrace = 1 \rbrace$.
	\end{enumerate}
\end{corollary} 
\subsection{Adjoining a normed linear space to an order unit space} 
Let $(V, e)$ be an order unit space and let $X$ be a real normed linear space. Consider $V_X := V \oplus_1 X$ and put $V_X^+ = \lbrace (v, x): \Vert x \Vert e \le v \rbrace$ and $e_X = (e, 0)$. It was shown in \cite{K23} that $(V_X, V_X^+, e_X)$ is an order unit space in such a way that the order unit norm coincides with the $\ell_1$-norm on $V_X$. Here we describe the canopy and its periphery corresponding to $V_X$. For this purpose, we introduce the following notion.
\begin{definition}
	Let $(V, e)$ be an order unit space. Then $u \in V^+$ is said to be a \emph{semi-peripheral} element if $u = \alpha (e - w) + (1 - \alpha) w$ for some $w \in R_V$ and $\alpha \in [0, 1]$. When $\alpha = \frac 12$, then $u = \frac 12 e$ which is called the \emph{central semi-peripheral} element. The set of all semi-peripheral elements of $V$ is denoted by $R_V^S$.
\end{definition} 
\begin{theorem}\label{sp}
	Let $(V, V^+, e)$ be an order unit space and let $X$ a real normed linear space. Consider the corresponding order unit space $(V \oplus_1 X, V_X^+, e_X)$. Then the canopy and the periphery of $V \oplus_1 X$ are given by 
	$$C_{V \oplus_1 X} = \lbrace (u, x) \in V_X^+: \Vert x \Vert e \le u ~ \mbox{and} ~ \Vert u \Vert + \Vert x \Vert = 1 \rbrace$$ 
	and 
	$$R_{V \oplus_1 X} = \lbrace (u, x): u \in R_V^S ~ \mbox{and} ~  \Vert x \Vert + \Vert u \Vert = 1 \rbrace.$$
\end{theorem}
\begin{proof}
	Let $w \in R_V$, $\alpha \in [0, 1]$ and put $u = \alpha (e - w) + (1 - \alpha) w$. Then $u \in V^+$ and $\Vert u \Vert = \max \lbrace \alpha, 1 - \alpha \rbrace$. Let $x \in X$ with $\Vert x \Vert = 1 - \Vert u \Vert$. For definiteness, we assume that $\alpha \ge \frac 12$. Then $\Vert u \Vert = \alpha$ and $\Vert x \Vert = 1 - \alpha$. Now, as 
	\begin{eqnarray*}
		u - \Vert x \Vert e &=& \alpha (e - w) + (1 - \alpha) w - (1 - \alpha) e \\ 
		&=& (2 \alpha - 1) (e - w) \in V^+
	\end{eqnarray*} 
	and $\Vert (u, x) \Vert_1 = 1$, we have $(u, x) \in C_{V \oplus_1 X}$.
	Further, as $e - u = (1 - \alpha) (e - w) + \alpha w$, we have $\Vert e - u \Vert = \alpha$ so that $\Vert e_X - (u, x) \Vert_1 = 1$. Thus $(u, x) \in R_{V \oplus_1 X}$. 
		
	Conversely, we assume that $(u, x) \in R_{V \oplus_1 X}$. Then $(u, x), (e - u, - x) \in V_X^+$ and we have $\Vert (u, x) \Vert_1 = 1 = \Vert (e - u, - x) \Vert_1$. Thus $\Vert x \Vert e \le u$, $\Vert - x \Vert e \le e - u$ 
	and 
	$$\Vert u \Vert + \Vert x \Vert = 1 = \Vert e - u \Vert + \Vert - x \Vert.$$
	Now, it follows that $\Vert u \Vert = \Vert e - u \Vert = 1 - \Vert x \Vert$ whence $(1 - \Vert u \Vert) e \le u$. Thus 
	\begin{eqnarray*}
		\Vert u - (1 - \Vert u \Vert) e \Vert &=& \sup \lbrace \phi\left(u - (1 - \Vert u \Vert) e\right): \phi \in S(V) \rbrace \\ 
		&=& \sup \lbrace \phi(u): \phi \in S(V) \rbrace - (1 - \Vert u \Vert) \\ 
		&=& 2 \Vert u \Vert - 1.
	\end{eqnarray*}
	If $2 \Vert u \Vert - 1 = 0$, that is $\Vert u \Vert = \frac 12$, then $u = \frac 12 e$ and is a (centrally) semi-peripheral element, for $(1- \Vert u \Vert) e \le u \le \Vert u \Vert e$. So we may assume that $2 \Vert u \Vert - 1 > 0$. Put $w = \frac{u - (1 - \Vert u \Vert) e}{2 \Vert u \Vert - 1}$. Then $w \in V^+$ and consequently, $w \in C_V$.
	
	Find $\phi \in S(V)$ such that $f\phi(e - u) = \Vert e - u \Vert = \Vert u \Vert$. Then $\phi(w) = 0$ so that $w \in R_V$. Further, we have $u = \Vert x \Vert (e - w) + (1 - \Vert x \Vert) w$ for $\Vert x \Vert = 1 - \Vert u \Vert$. Thus $u$ is again a semi-peripheral element.
\end{proof}
\begin{remark}
	We can deduce from the proof of Theorem \ref{sp} that 
	$$R_V^S = \lbrace u \in [0, e]: \Vert u \Vert = \Vert e - u \Vert \rbrace = \lbrace u \in V: \Vert u \Vert = \Vert e - u \Vert \le 1 \rbrace.$$
\end{remark}

\section{Some applications}

\begin{lemma}\label{l2}
	Let $(V, e)$ be an order unit space of dimension $\ge 2$ (so that $R_V \ne \emptyset$). Let $u \in R_V$ and consider 
	$$P_u := \lbrace \alpha e + \beta u: \alpha, \beta \in \mathbb{R} \rbrace.$$ 
	Then $P_u$ is a unitally order isomorphic to $\ell_{\infty}^2$.
\end{lemma}
\begin{proof}
	Consider the mapping $\chi: P_u \to \ell_{\infty}^2$ given by $\chi(\alpha e + \beta u) = (\alpha, \alpha + \beta)$ for all $\alpha, \beta \in \mathbb{R}$. Then $\chi$ is a unital bijection. We show that $\chi$ is an order isomorphism. We first assume that $\alpha e + \beta u \in V^+$. Since $u \in R_V$, we can find $f, g \in S(V)$ such that $f(u) = 1$ and $g(u) = 0$. Thus $0 \le f(\alpha e + \beta u) = \alpha + \beta$ and $0 \le g(\alpha e + \beta u) = \alpha$. Thus $(\alpha, \alpha + \beta) \in \ell_{\infty}^{2 +}$. Conversely, we assume that $(\alpha, \alpha + \beta) \in \ell_{\infty}^{2 +}$. Then $\alpha e + \beta u = \alpha (e - u) + (\alpha + \beta) u \in V^+$. 
\end{proof}
\begin{remark}\label{norm}
	As $u \in R_V$, we have $u \perp_{\infty} (e - u)$. Thus, it is simple to show that $\chi$ is an isometry. In fact, if $v \in P_u$, say $v = \alpha e + \beta u = \alpha (e - u) + (\alpha + \beta) u$,  then  
	$$\Vert v \Vert = \Vert \alpha (e - u) + (\alpha + \beta) u \Vert = \max \lbrace \vert \alpha \vert, \vert \alpha + \beta \vert \rbrace.$$ 
\end{remark} 

\begin{theorem}
	Let $(V, e)$ be an order unit space with $\dim(V) \ge 2$. Then $V$ contains a copy of $\ell_{\infty}^2$ as an order unit subspace. Moreover, $V$ contains a copy of $\ell_{\infty}^n$ ($n \ge 2$) as an order unit subspace if and only if there exist $u_1, \dots, u_n \in R_V$ such that $u_i \perp_{\infty} u_j$ for all $i, j \in \lbrace 1, \dots, n \rbrace$ with $i \ne j$, $\sum_{i=1}^{n} u_i = e$ and $\perp_{\infty}$ is additive in the linear span of $u_1, \dots, u_n$. 
\end{theorem} 
\begin{proof}
	It follows from Lemma \ref{l2} that $V$ contains a copy of $\ell_{\infty}^2$ as an order unit subspace. 
	
	Next, we assume that $W$ is an order unit subspace of $V$ and $\Gamma: \ell_{\infty}^n \to W$ is a surjective unital order isomorphism. Put $\gamma(e_i) = u_i$ for $i = 1, \dots, n$ where $\lbrace e_1, \dots, e_n \rbrace$ is the standard unit basis of $\ell_{\infty}^n$. Then $u_1, \dots, u_n \in C_V$ with $\sum_{i=1}^{n} u_i = e$. Consider the biorthonormal system $\lbrace f_1, \dots, f_n \rbrace$ in $\ell_1^n$ so that $f_i(e_j) = \delta_{ij}$. Then $\lbrace f_1 \circ \Gamma^{-1}, \dots, f_n \circ \Gamma^{-1} \rbrace$ is the set of pure states of $W$. We can extend $f_i \circ \Gamma^{-1}$ to a pure state $g_i$ of $V$ for each $i = 1, \dots, n$. Then $g_i(u_j) = \delta_{ij}$ so that $u_1, \dots, u_n \in R_V$ and we have $u_i \perp_{\infty} u_j$ for all $i, j \in \lbrace 1, \dots, n \rbrace$ with $i \ne j$. Also, if $\alpha_1, \dots, \alpha_n \in \mathbb{R}$, then 
	\begin{eqnarray*}
		\left\Vert \sum_{i=1}^{n} \alpha_i u_i \right\Vert &=& \left\Vert \sum_{i=1}^{n} \alpha_i e_i \right\Vert_{\infty}\\ &=& \max \lbrace \vert \alpha_i \vert: 1 \le i \le n \rbrace\\ &=& \max \lbrace \Vert \alpha_i u_i \Vert: 1 \le i \le n \rbrace.
	\end{eqnarray*} 
	Thus $\perp_{\infty}$ is additive in the linear span of $u_1, \dots, u_n$. 
	
	Conversely, we assume that there exist $u_1, \dots, u_n \in R_V$ such that $u_i \perp_{\infty} u_j$ for all $i, j \in \lbrace 1, \dots, n \rbrace$ with $i \ne j$, $\sum_{i=1}^{n} u_i = e$ and $\perp_{\infty}$ is additive in the linear span $U$ of $u_1, \dots, u_n$. Define $\Phi: U \to \ell_{\infty}^n$ by $\Phi( \sum_{i=1}^{n} \alpha_i u_i) = (\alpha_i)$. Then $\Phi$ is a unital linear bijection. Also 
	\begin{eqnarray*}
		\left\Vert \Phi \left( \sum_{i=1}^{n} \alpha_i u_i \right) \right\Vert &=& \Vert (\alpha_i) \Vert_{\infty} \\ &=& \max \lbrace \vert \alpha_i \vert: 1 \le i \le n \rbrace\\ &=& \max \lbrace \Vert \alpha_i u_i \Vert: 1 \le i \le n \rbrace \\ &=& \left\Vert \sum_{i=1}^{n} u_i e_i \right\Vert
	\end{eqnarray*} 
	as $\perp_{\infty}$ is additive on $U$. Thus that $\Phi$ is an isometry. Now, being a unital linear surjective isometry, $\Phi$ is a unital order isomorphism.
\end{proof}
\begin{remark}
	An order unit space of dimension $2$ is unitally isometric to $\ell_{\infty}^2$. However, we shall show in a forthcoming paper that for the dimension greater than $2$, there exist non-isometric order unit spaces of the same dimension.
\end{remark}
\begin{proposition}\label{disj} 
	Let $(V, e)$ be an order unit space and let $u, v \in R_V$. Then $u \ne v$ if and only if $[e, u] \bigcap [e, v] = \lbrace e \rbrace$.
\end{proposition}
\begin{proof}
	First, let $w \in [e, u] \bigcap [e, v]$ with $w \ne e$.As $w \in [e, u] \bigcap [e, v]$, we can find $0 < \lambda, \mu \le 1$ such that $w = e - \lambda (e - u) = e - \mu (e - v)$. Then $\lambda (e - u) = \mu (e - v)$. Since $u, v \in R_V$, we have $e - u, e - v \in C_V$. Thus $\Vert e - u \Vert = 1 = \Vert e - v \Vert$ so that $\lambda = \mu$. As $\lambda, \mu > 0$, we get $u = v$. Thus $u \ne v$ implies $[e, u] \bigcap [e, v] = \lbrace e \rbrace$. Evidently, $u = v$ implies $[e, u] \bigcap [e, v] = [e, u] \ne \lbrace e \rbrace$. 
\end{proof}
\begin{remark}
	We note that $C_V = \bigcup_{u \in R_V} [0, u]$ is a disjoint union of \emph{untangled strings} $[0, u]$'s of $K_V$ attached to $e$. Consequently, $e - C_V = \bigcup_{u \in R_V} [0, u]$ is a disjoint union of \emph{untangled strings} $[0, u]$'s of $K_V$ attached to $0$ as well. 
\end{remark}
\begin{proposition}\label{uniqueu}
	Let $(V, e)$ be an order unit space and let $u, v \in R_V$. Then, either $P_u = P_v$ or $P_u \bigcap P_v = \mathbb{R} e$.
\end{proposition}
\begin{proof}
	Let $w \in P_u \bigcap P_v$ be such that $w \notin \mathbb{R} e$. Without any loss of generality, we assume that $\Vert w \Vert = 1$. Consider $w_1 := \frac 12 (e + w)$. Then $w_1 \in V^+$. Also then $w_1 \in P_u \bigcap P_v$. So, without any loss of generality again, we further assume that $w \in V^+$, that is, $w \in C_V$. 
	
	Since $w \in P_u$, we have $w = \alpha e + \beta u$ for some $\alpha, \beta \in \mathbb{R}$. As $w \in C_V$, we have $\alpha \ge 0$, $\alpha + \beta \ge 0$ and $\max \lbrace \alpha, \alpha + \beta \rbrace = 1$. Since $w \notin \mathbb{R} e$, we have $\beta \ne 0$. If $\beta > 0$, then we have $1 = \alpha + \beta > \alpha \ge 0$. Thus $w = \alpha e + (1 - \alpha) u \in [e, u]$. Next if $\beta < 0$, then $1 = \alpha > \alpha + \beta \ge 0$ so that $- 1 \le \beta < 0$. Thus $w = e + \beta u = e - (- \beta) (e - (e - u)) \in [e, e - u]$. Summing up, we have $w \in [e, u] \bigcup [e, e - u]$. Similarly, as $w \in P_v$, we also have $w \in [e, v] \bigcup [e, e - v]$. Thus $w \in \left( [e, u] \bigcup [e, e - u] \right) \bigcap \left( [e, v] \bigcup [e, e - v] \right)$. Since $w \ne e$, using Proposition \ref{disj}, we conclude that one of the equalities $[e, u] = [e, v]$ or $[e, u] = [e, e - v]$, or $[e, e - u] = [e, v]$, or $[e, e - u] = [e, e - v]$ hold. In other words, either $u = v$ or $u = e - v$. In both the situations, we have $P_u = P_v$. 
\end{proof} 
\begin{theorem}\label{ourep}
	Let $(V, e)$ be an order unit space with $\dim V \ge 2$. Then $V = \bigcup \lbrace P_u: u \in R_V \rbrace$ in such a way that $\bigcap \lbrace P_u: u \in R_V \rbrace = \mathbb{R} e$ and if $v \in V$ with $v \notin \mathbb{R} e$, then there exists a unique $w \in R_V$ such that $v \in P_w = P_{(e - w)}$.
\end{theorem} 
\begin{proof}
	Let $v \in V$ with $v \notin \mathbb{R} e$. For simplicity, we assume that $\Vert v \Vert = 1$. Put $v_1 = \frac 12 (e + v)$ and $v_2 = \frac 12 (e - v)$. Then $v_1, v_2 \in V^+$ and $v_1, v_2 \notin \mathbb{R} e$. Also then $v = v_1 - v_2$ and $1 = \Vert v \Vert = \max \lbrace \Vert v_1 \Vert, \Vert v_2 \Vert \rbrace$. Replacing $v$ by $- v$, if required, we further assume that $0 < \Vert v_2 \Vert \le \Vert v_1 \Vert = 1$, that is, $\Vert e + v \Vert = 2$. Thus we can find $f \in S(V)$ such that $1 + f(v) = f(e + v) = 2$. Therefore, $f(v) = 1$. Put $w = \Vert e - v \Vert^{-1} (e - v)$. Then $w \in C_V$. Further, $f(e - w) = 1 - \Vert e - v \Vert^{-1} (1 - f(v)) = 1$ so that $e - w \in C_V$. Thus $w \in R_V$. Now $v = e - \Vert e - v \Vert w$ so that $v \in P_w = P_{(e - w)}$.  
	
	Uniqueness of $w$ follows from Proposition \ref{uniqueu}.
\end{proof}
\begin{remark}
	Let $v \in V$ with $v \notin \mathbb{R} e$. Find $f \in S(V)$ such that $\Vert v \Vert = \vert f(v) \vert$. If $\Vert v \Vert = f(v)$, then $\Vert e + \Vert v \Vert^{-1} v \Vert = 2$. Thus $v$ has a unique representation $v = \lambda e + \mu w$ in $P_w$ where $w = \Vert \Vert v \Vert e - v \Vert^{-1} (\Vert v \Vert e - v) \in R_V$ and $\lambda = \Vert v \Vert$ and $\mu = - \Vert \Vert v \Vert e - v \Vert$. When $\Vert v \Vert = - f(v)$, we replace $v$ by $- v$. 
\end{remark}

\thanks{{\bf Acknowledgements:} 
	The author was partially supported by Science and Engineering Research Board, Department of Science and Technology, Government of India sponsored  Mathematical Research Impact Centric Support project (reference no. MTR/2020/000017).}


\begin{thebibliography}{100}

\bibitem{A71} E. M. Alfsen, \textit{Compact convex sets and boundary integrals}, Springer-Verlag, 1971.

\bibitem{A68} L. Asimov, Well-capped convex cones, Pacific J. Math., \textbf{26} (1968), 421-431. 

\bibitem{B55} F. F. Bonsall, \textit{Endomorphisms of a partially ordered space without order unit}, J. Lond. Math. Soc., \textbf{30} (1955), 144-153. 

\bibitem{CE77} M. D. Choi and E. G. Effros, \textit{Injectivity and operator spaces}, J. Funct. Anal., \textbf{24} (1977), 156-209. 

\bibitem{E64} D. A. Edwards, \textit{Homeomorphic affine embedding of a locally compact cone into a Banach space endowed with vague topology}, Proc. Lond. Math. Soc, \textbf{14} (1964), 399-414.

\bibitem{El64} A. J. Ellis, \textit{The duality of partially ordered normed linear spaces}, J. Lond. Math. Soc., \textbf{39} (1964), 730-744. 

\bibitem{GK20} A. Ghatak and A. K. Karn, \textit{Quantization of $A_0(K)$ spaces}, Operator and Matrices, 14(2) (2020), 381-399. 

\bibitem{Jam70} G. J. O. Jameson, \textit{Ordered linear spaces}, Lecture Notes in Mathematics, Springer-Verlag, 1970.

\bibitem{K51a} R. V. Kadison, \textit{Order properties of bounded self-adjoint operators}, Proc. Amer. Math. Soc., \textbf{2}(3) (1951), 505-510. 

\bibitem{K51b} R. V. Kadison, \textit{A representation theory for commutative topological algebras}, Mem. Amer. Math. Soc., \textbf{7}(1951). 

\bibitem{K14} A. K. Karn, \textit{Orthogonality in $\ell_p$ -spaces and its bearing on ordered Banach spaces}, Positivity \textbf{18} (2014), no. 2, 223–234.

\bibitem{K19} A. K. Karn, \textit{Algebraic orthogonality and commuting projections in operator algebras} Acta. Sci. Math. (Szeged), \textbf{84} (2018), 323--353.

\bibitem{K22} A. K. Karn, \textit{Orthogonality: An antidote to Kadison's anti-lattice theorem}, Positivity and its Applications, Trends Math., Birkhäuser/Springer, Cham, (2021), 217-227. 

\bibitem{K23} A. K. Karn, \textit{Order units in normed linear spaces}, Preprint 11 pages. https://doi.org/10.48550/arXiv.2306.06549. 

\bibitem{N69} K.-F. Ng, \textit{The duality of partially ordered Banach spaces}, Proc. Lond. Math. Soc., \textbf{19} (1969), 268-288. 

\bibitem{S51} S. Sherman, \textit{Order in operator algebras}, Amer. J. Math., \textbf{73}(1) (1951), 227-232.

\end{thebibliography}
\end{document}